\documentclass{amsart}%
\usepackage{amssymb,mathabx}
\usepackage{color}
\usepackage{amsmath}
\usepackage{graphicx}
\usepackage{amsfonts}%
\usepackage[colorlinks=false]{hyperref}
\usepackage{setspace}
\usepackage{etoolbox}
\usepackage{tikz-cd}
\usepackage{enumerate}
\usepackage{adjustbox}
\usepackage{multicol}
\usepackage{doi}

\usetikzlibrary{matrix,arrows,decorations.pathmorphing,decorations.markings} 

\usepackage[pagewise]{lineno}

\newcommand{\F}{\mathbb F}
\newcommand{\R}{\mathbb R}
\newcommand{\C}{\mathbb C}

\newcommand{\Q}{\mathbb Q}

\newcommand{\Zd}{\mathbb{Z}/2\mathbb{Z}}
\newcommand{\zd}{\mathbb{Z}/2\mathbb{Z}}

\usepackage{bm}
\usepackage{amsfonts, stmaryrd}
\usepackage[bbgreekl]{mathbbol}
\DeclareSymbolFontAlphabet{\mathbbm}{bbold}
\DeclareSymbolFontAlphabet{\mathbb}{AMSb}
\usepackage{mathrsfs}
\usepackage{tikz-cd}
\usetikzlibrary{matrix,arrows,decorations.pathmorphing,decorations.markings} 
\usepackage{pb-diagram}

\newtheorem{theorem}{Theorem}[section]

\newtheorem{definition}[theorem]{Definition}
\newtheorem{example}[theorem]{Example}

\newtheorem{lemma}[theorem]{Lemma}

\newtheorem{proposition}[theorem]{Proposition}
\newtheorem{remark}[theorem]{Remark}

\numberwithin{equation}{section}
\AtBeginEnvironment{equation}{\leavevmode\singlespace}
\AfterEndEnvironment{equation}{\endsinglespace\vskip0.5\baselineskip\noindent\ignorespaces}

\newcommand{\bk}{\Bbbk}
\usepackage{twoopt}

\newcommandtwoopt\eck[3][G][\Bbbk]{H^*_{#1}(#3; #2)}

\newcommandtwoopt\csk[2][G][\Bbbk]{H^*(B{#1}; #2)}

\newcommand\inv[1]{#1^{-1}}
\DeclareMathOperator{\ext}{Ext} 
\DeclareMathOperator{\rank}{rank} 

\DeclareMathOperator{\depth}{depth}

\title[Equivariant cohomology for elementary $2$-abelian groups]{
The quotient criterion for syzygies in equivariant cohomology for elementary abelian $2$-group actions
}
\date{ \today }

\author[Chaves]{Sergio Chaves}
\address[Sergio Chaves]{Department of Mathematics\\Western University \\London, ON. Canada}
\email{schavesr@uwo.ca}

\begin{document}

\begin{abstract}
Let $G$ be a elementary abelian $2$-group and $X$ be a manifold with a locally standard action of $G$. We provide a criterion to determine the syzygy order of the $G$-equivariant cohomology of $X$ with coefficients over a field of characteristic two using a complex associated to the cohomology of the face filtration of the manifold with corners $X/G$. This result is the real version of the quotient criterion for locally standard torus actions developed in \cite{franzQuot}.
\end{abstract}

\maketitle

\section{Introduction}\label{se:intro}

Let $G$ be a compact Lie group and $X$ be a $G$-space. The $G$-equivariant cohomology of $X$ with coefficients over a field $\bk$ is defined as the singular cohomology of the Borel construction \cite{borel1961seminar} $H^*_G(X; \bk)=H^*(X_G; \bk)$. It inherits a canonical module structure over  the cohomology of the classifying space $H^*(BG;\bk)$ of $G$.

If the restriction map $H^*_G(X;\bk) \rightarrow H^*(X;\bk)$ is surjective, the $G$-equivariant cohomology is a free module over $H^*(BG;\bk)$ as a consequence of the Leray-Hirsch theorem. In this case, we say that $X$ is $G$-equivariantly formal over $\bk$. The converse also holds when extra assumptions over $G$ are considered. For example,  if $G$ is connected, it is a direct consequence of the Eilenberg-Moore spectral sequence associated to the fibration $X \rightarrow X_G \rightarrow BG$, and a discussion for nilpotent actions of any $G$ can be found in \cite[\S 4.1]{ptori}.  Examples of equivariantly formal spaces over $\Q$ include smooth compact toric varieties and quasitoric manifolds \cite{davis1991convex}, symplectic manifolds with Hamiltonian torus actions \cite{atiyah1982convexity} and  $G$-spaces with vanishing odd rational cohomology when $G$ is connected. Moreover, when cohomology with rational coefficients is considered, the ring $H^*(BG;\Q)$ becomes a polynomial ring in $n$ generators sitting in even degrees. In this case, a $G$-space $X$ is equivariantly formal if and only if its equivariant cohomology fits in a long exact sequence
\begin{equation}
0 \rightarrow H^*_G(X;\Q) \rightarrow F_1 \rightarrow \cdots F_n
\end{equation}
of free $H^*(BG;\Q)$-modules $F_j$ for $1 \leq j \leq n$ by the Hilbert Syzygy Theorem. This equivalence motivates the study of syzygies in equivariant cohomology started in \cite{AFP} for torus actions.

Recall that a finitely generated module $M$ over a commutative ring $R$ is a $j$-th syzygy if there is an exact sequence
\begin{equation}\label{eq:syzygy}
0 \rightarrow M \rightarrow F_1 \rightarrow \cdots \rightarrow F_j
\end{equation}
of free modules $F_k$ for $1 \leq k \leq j$. In \cite[Thm.5.7]{AFP}, the authors showed that the syzygy order of the equivariant cohomology of a space with a torus action is equivalent to the partial exactness of the Atiyah-Bredon  with rational coefficients.  This sequence  was firstly discussed in \cite{atiyah2006elliptic} \cite{bredon1974free} and it is defined in the following way: Let $T = (S^1)^n$ be a torus and $X$ be a $T$-space. The filtration of $X$ by its orbit dimensions $X_0 = X^T \subseteq X_1 \subseteq \cdots \subseteq X_n = X $ induces a complex 
\begin{equation}\label{abseq}
 0 \rightarrow H^*_T(X) \rightarrow H^*_T(X^T) \rightarrow H^{*+1}(X_1,X_0) \rightarrow \cdots \rightarrow H^{*+n}(X_n, X_{n-1}) 
\end{equation}
that is referred nowadays as the Atiyah-Bredon sequence of the $T$-space $X$.

The characterization of syzygies via the partial exactness of the sequence \ref{abseq} can be extended to any compact connected Lie group $G$ over the rationals by restriction of the action to a maximal torus  $T\subseteq G$ \cite{franz2016syzygies} and to any elementary abelian $p$-group  $G_p$ over a field of characteristic $p$ by transfer and restriction of the action under the inclusion $G_p \subseteq T$ \cite{ptori}. 

Another remarkable characterization of syzygies in equivariant cohomology for torus actions is the quotient criterion for locally standard actions developed in \cite{franzQuot}. Recall that these actions are modeled by the standard representation of $T=(S^1)^n$ on $\C^n$. Such characterization is given as follows: for  a locally standard smooth action on a $T$-manifold $X$, the quotient space $M = X/T$ is a nice manifolds with corners and the syzygy order of the $T$-equivariant cohomology of $X$ it is determined by the topology of filtration of $M$ by its faces. In particular,  this result recovers the equivariant formality over $\Q$ of compact smooth toric varieties, torus manifolds and quasi-toric manifolds.

In this paper, we discuss locally standard torus actions modeled by the standard representation of $G = (\zd)^n$ on $\R^n$ and cohomology with coefficients over a field of characteristic 2, that we refer as the ``real version'' of the torus case.  We use the characterization of syzygies via the Atiyah-Bredon sequence for elementary abelian $2$-groups and the description of the Ext modules of the equivariant homology as the cohomology of the Atiyah-Bredon complex as discussed in \cite{ptori} analogously to the torus actions case \cite{franzQuot}. The main result of this document is the following.

Let $G = (\zd)^n$ be a elementary abelian $2$-group and let $X$ be a compact manifold of dimension $m \geq n$ with a locally standard action of $G$. Then $M = X/G$ becomes a $m$-manifold with $n$-corners (Definition \ref{def:mc}) and for any face $P \subseteq M$ we consider the complex
 \[ B^i(P)  = \bigoplus_{\substack{Q \subseteq P \\ \rank Q = i }} H^{*}(Q, \partial Q) \]
with differential induced by the connecting homomorphism of the cohomological long exact sequence associated to the triple $(Q,\partial Q, \partial Q \setminus (P \setminus \partial P) )$. The syzygy order of the $G$-equivariant cohomology of $X$ is determined by the vanishing of the cohomology of the complex $B^i(P)$ for any $P$ in certain range as we state in the following theorem.

\begin{theorem}
	 Let $\bk$ be a field of characteristic two and $1 \leq j \leq n$. $H^*_G(X; \bk)$ is a $j$-th syzygy over $H^*(BG;\bk)$ if and only if for any face $P$ of the manifold with corners $M = X/G$ we have that $H^i(B^*(P))= 0$ for any $i > \max(\rank P - j, 0)$.
\end{theorem}	

As consequence of this result, we immediately recover the equivariant formality for $G$-spaces whose orbit space and its faces are  contractible; for example, the real locus of quasitoric manifolds whose orbit space is a simple polytope. This also provides a criterion to compute the syzygy order of the $(\zd)^n$-manifolds constructed in \cite{lu2008}, \cite{yu2012}.

This document is organized as follows: In section \ref{sec:1} we review equivariant cohomology for elementary abelian 2-groups (or 2-tori) and provide a characterization of syzygies in terms of decomposition of the subgroups of the 2-torus. In Section \ref{sec:2} we review the concept of manifolds with corners and in Section \ref{sec:3} we discuss locally standard 2-torus actions, provide a proof of the main results and discuss some consequences of them.

\textit{Acknowledgments.}  I would like to thank Matthias Franz for his ideas and collaboration on this work, as well as his comments on earlier versions of this document. I am also grateful to Christopher Allday, Matthias Franz and Volker Puppe for sharing an earlier copy of \cite{ptori} with me. This work is based on the author's doctoral thesis.

\section{Remark on syzygies in equivariant cohomology for 2-torus actions}\label{sec:1}

In this section, we review the characterization of syzygies in equivariant cohomology for actions of a group  $G \cong (\Zd)^n$ isomorphic to a $2$-torus of rank $n$ and cohomology with coefficient over a field of characteristic two $\bk$. We will omit the coefficient $\bk$ in our notation for cohomology. 

We start by reviewing the construction of the Atiyah-Bredon sequence for $2$-torus actions. The relation between syzygies in $G$-equivariant cohomology and the Atiyah-Bredon sequence, the $G$-equivariant homology and the equivariant Poincar\'e duality  has been developed in \cite{ptori}  where the authors generalize analogous results from the torus case \cite{AFP}. 
Let $G$ be a 2-torus of rank $n$ and $X$ be a $G$-space. The $i$-th $G$-equivariant skeleton of $X$ is the space $X_i$ defined as the union of orbits  of size at most $2^i$ for $-1 \leq i \leq n$. The skeletons of $X$ give rise to a filtration
\[ \emptyset = X_{-1} \subseteq X_0 \subseteq \cdots \subseteq X_r = X  \]
called the \textit{$G$-orbit filtration} of $X$. This filtration induces a complex 
\[ 0 \rightarrow H^*_G(X) \rightarrow  H_G^*(X_0) \rightarrow H_G^{*+1}(X_1, X_0) \rightarrow \cdots \rightarrow H_G^{*+n}(X_n, X_{n-1})\]
which is called the \textit{$G$-Atiyah-Bredon sequence} of $X$ and it will be denoted by $AB_G^*(X)$. We will show a characterization of syzygies in terms of the exactness of the Atiyah-Bredon sequence $AB_L(X^K)$ for any decomposition $G \cong K \times L$ analogously to \cite[\S 3]{franzQuot}. As the $G$-equivariant cohomology is a module over the polynomial ring $H^*(BG)$, we first review the following algebraic remark.

\begin{remark}\label{algrmk}
	Let $S$ be a polynomial ring over some field in $n$-variables of positive degree, and let $\mathfrak{m}$ be the maximal homogeneous ideal of $S$. For a graded finitely generated $S$-module $M$, the length of a maximal $M$-sequence of elements in $\mathfrak{m}$  is denoted by $\depth_S M$. It is related to the Ext functor via the formula \[\depth_S M =\min \{ k : \ext_S^{n-k}(M,S)\}. \]
	See \cite[Prop.A1.16]{eisenbud}. On the other hand, the depth of $M$ and the syzygy order of $M$ are related as follows: $M$ is a $j$-th syzygy over $S$ if and only if for any prime ideal $\mathfrak{p}\subseteq S$, $\depth_{S_\mathfrak{p}} M_\mathfrak{p} \geq \min(j, \dim S_\mathfrak{p})$ \cite[\S 16.E]{bruns2006determinantal}.
\end{remark}
Now we proceed to prove the following result.

\begin{lemma}\label{lem434}\quad
	Let $X$ be a $G$-space. The $G$-equivariant cohomology $H^*_G(X)$ is a $j$-th syzygy over  $R$ if and only if for any decomposition $G \cong K \times L$ into two tori $K$ and $L$ it holds that
	\[\depth_{R_L} H_L^*(X^K)\geq \min(j, \rank L).\]
	\begin{proof}
		The proof given in \cite[Prop.3.3]{franzQuot} is purely algebraic and carries over to our setting. It mainly uses the characterization of syzygies via depth as in Remark \ref{algrmk}, and that one can restrict only to the prime ideals corresponding to those arisen as the kernel of the restriction map $H^*(BG)\rightarrow H^*(BK)$ for any subgroup $K \subseteq G$.  Compare with \cite[Lem.8.1]{ptori}.
	\end{proof}
\end{lemma}
The next proposition uses the  $G$-equivariant homology of $X$ and the equivariant extension of the Poincar\'e duality for $2$-torus actions. See \cite{ptori} for a wider discussion on equivariant homology and the equivariant Poincar\'e duality.

\begin{proposition}\label{depthcor}
	Let $X$ be a $G$-manifold and $j\geq 0$. Then $H^*(X)$ is a $j$-th syzygy if and only if $H^i(AB^*_L(X^K))= 0$ for any subgroup $K$ occurring as an isotropy subgroup in $X$  where $L = G/K$ and $i > \max(\rank L - j, 0)$.
\end{proposition}
\begin{proof}
	The proof from the torus case \cite[Cor.3.4]{franzQuot} can be also adapted to our setting. As discussed in \cite[Thm.8.3]{ptori}, we may also assume that $\bk = \F_2$. Let $K$ be a subgroup of $G$. Then $X^K$ is a closed submanifold of $X$ by the tubular neighbourhood theorem. For a connected component $Y \subseteq X^K$, there is  a principal orbit $G/G_x$ where $x \in Y$, so $K \subseteq G_x$ as subgroup. Set $K' = G_x$. $L'= G/K'$ and write $\rank K' = \rank K + k$ for some integer $k \geq 0$. Using Lemma \ref{lem434} we get 
	\begin{align*}
	\depth_{R_L} H^*_L(Y) &\geq \depth_{R_L} H^*_L(X^{K'}) \\
	&= \depth_{R_L'} H^*_{L'}(X^{K'}) + k \\
	&\geq \min(j,\rank L')+k \geq \min(j,\rank L).
	\end{align*}
	Following \cite[Thm.8.9]{ptori} and analogous to \cite[Thm.4.8]{AFP} for the torus case, we have that the cohomology of the $G$-Atiyah-Bredon complex of $X$ is isomorphic to the ext of the equivariant homology of $X$; namely, $H^i(AB^*_G(X)) \cong \ext^i_R(H_*^G(X),R) $ for any $i \geq 0$. This implies that
	\begin{align*}
	\depth_{R_L} H^*_L(X^K) &= \min \{ i : \ext_{R_L}^{\rank L  - i}(H^*_L(X^K), R_L) \neq 0\} \\
	&= \min \{ i : H^{\rank L  - i}(AB^*_L(X^K)) \neq 0\}
	\end{align*}
	by combining also the equivariant Poincar\'e duality isomorphism $H_*^G(X)\cong  H^*_G(X)$.
	
	\medskip
	
	Therefore, $\depth_{R_L} H^*_L(X^K) \geq \min(j,\rank L)$ if and only if $H^i(AB^*_L(X^K))= 0\text{ for all } i > \max(\rank L -j,0). $ \qedhere
\end{proof}
We finish this section with the following result (compare with \cite[Prop.3.3]{franzQuot}).
\begin{proposition}\label{teo2.1}
	If $H^*_G(X)$ is a $j$-th syzygy over $R$ then so is $H^*_L(X^K)$ over $R_L$ for any subgroup $K \subseteq G$ and complementary subgroup $L \subseteq G$. Furthermore, $L$ can be canonically identified with the quotient $G/K$.
\end{proposition}
\begin{proof}
	Let $K \subseteq G$, $Y = X^K$ and $L$ be a complementary subgroup to $K$. We will show that the condition of Lemma \ref{lem434} holds for $H^*_L(Y)$. Let $K' \subseteq L$ and choose a complementary subgroup  $L' \subseteq L$ of $K'$ in $L$. Notice that $K_0 \cong K \times K'$ is a complementary subgroup of $L'$ in $G$, and we have that $Y^{K'} = (X^K)^{K'} = X^{K_0}$.
	%
	Applying then Lemma \ref{lem434} to the subgroup $K_0 \subseteq G$, we get that
	\begin{align*}
	\depth_{R_{L,L'}}H^*_{L'}(Y^{K'}) &= \depth_{R_{L'}}H^*_{L'}(X^{K_0}) \geq \min(j,\rank L')
	\end{align*}
	showing that $H^*_L(Y)$ is a $j$-th syzygy over $R_L$ again by Lemma \ref{lem434}.
\end{proof}

\section{Remark on manifolds with corners}\label{sec:2}

In this section we review the notion of manifolds with corners that generalizes the concept of manifolds and manifolds with boundary in the classical setting. They were firstly developed in \cite{cerf} and \cite{douady} in differential geometry and have  been used in transformation groups on smooth manifolds \cite{janich}, cobordism  \cite{laures} and toric topology \cite{taraspanov}.

Recall that a topological manifold of dimension $n$ is locally modeled by $\R^n$, a manifold with boundary is modeled by the half space $  [0,\infty]\rtimes\R^{n-1}$ and a manifold with corners with be modeled by the intersection of (zero or more) half spaces in $\R^n$ as we state in the following definition.

\begin{definition}\label{def:mc}
	Let $M$ be a paracompact Hausdorff space and $m \geq n \geq 0$ be integers. We say that {$M$ is an $m$-manifold with $n$-corners} if $M$ has an atlas $\{(U_i, \varphi_i )\}$ where $\varphi_i \colon U_i \rightarrow V_i$ is a homeomorphism of $U_i$ onto an open subset $V_i$ or $R_{m,n} := [0, \infty)^n \times \R^{m-n}$, and the map $\varphi_i\circ \varphi_j^{-1} \colon \varphi_j(U_i \cap U_j) \rightarrow \varphi_i(U_i \cap U_j)$ is the restriction  of a diffeomorphism between open sets in $\R^m$ for all $i,j$.
\end{definition}
Even though we provide a general definition, in this document we will only discuss compact spaces for simplicity. Examples of manifolds with corners include manifolds, manifolds with boundary and convex simple polytopes. More interesting examples of spaces that are manifolds with corners are given by the following figures: the teardrop and the eye-shaped space.

\begin{figure}[h]\label{fig1}
	\centering
	\includegraphics[width=0.5\textwidth]{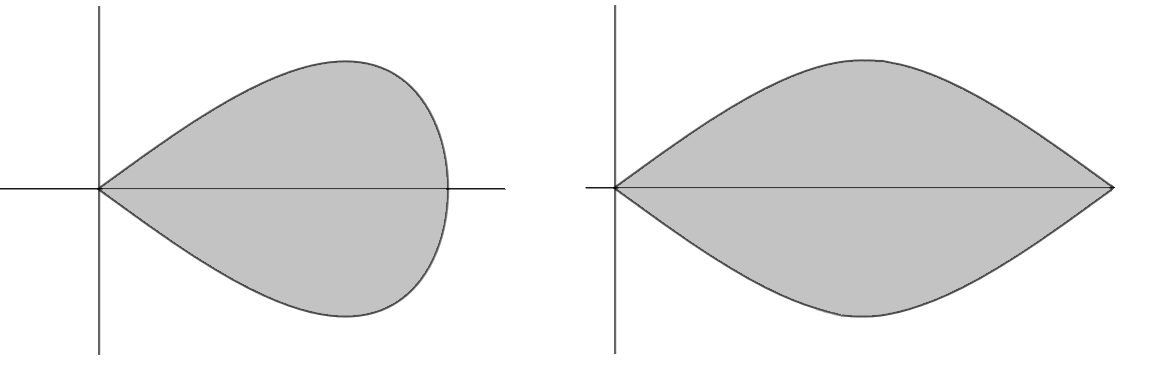}
	\caption{Examples of manifold with corners: Teardrop and Eye-shaped figure}
\end{figure}

For any $m$-manifolds with $n$-corners $M$, the boundary $\partial M$ becomes a $(m-1)$-manifold with $(n-1)$-corners. Moreover, we can filter $M$ in the following way. For any $z = (x,y)  \in \R_{m,n}$, let $c_z$ be the number of zero coordinates of $x$ in $[0, \infty)^n$. If $M$ is an $m$-manifold with $n$-corners, then $c_x$ is well-defined for any $x \in M$. We say that $F$ is a facet of $M$ if $F$ is the closure  of a connected component of the subspace $M_1 = \{ x \in M \colon c_x = 1 \}$. Notice that $F$ is an $(m-1)$-dimensional submanifold with boundary of $\partial M$ and $\bigcup_{F\,\text{facet}} F = \partial M$.   Moreover, any finite intersection of facets $\bigcap_{i =1}^k F_i$ is either empty, or a disjoint union of  submanifold of $M$ of codimension $k$. Analogously, a face of $M$ of codimension $k$ is defined as the closure of a connected component of the subspace $M_{k} = \{ x \in M : c_x = k \}$. 
\begin{remark}
	Any manifold with corners become a filtered space by setting $X_i = \bigcup_{k \leq i} M_{n-i}$, so $X_i$ consists of all faces of codimension at least $n-i$. In particular,  $X_0 = M_n$ and $X_n = M$.
\end{remark}

For example, a face of $\R_{m,n} = [0, \infty)^n \times \R^{m-n}$ of codimension $k$ is the subspace $A_I = \{ (x,y) \in \R_{m,n} : x_i = 0\  \text{for\;} i \notin I \}$ for some $I \subseteq \{1,\ldots, n\}$, $|I|= k$. 

\begin{definition}
Let $M$ be a manfold with corners. We say that $M$ is a nice manifold with corners if for any face $P$ of $M$ there are exactly $k$-facets $F_1,\subseteq, F_k$ such that $P$ is a connected component of the intersection $\bigcap_{i=1}^k F_i$.
\end{definition}

For example, in the spaces from  Figure \ref{fig1}, we have that the teardrop is a 2-manifold with 2-corners that is not nice but the eye-shaped figure it is a nice 2-manifolds with 2-corners. For our interests, we will focus on nice manifolds with corners as they generalize the notion of simple polytopes in toric topology as they will show up as the orbit space of locally standard actions as discussed in the following section.

\section{The quotient criterion for locally standard 2-torus actions}\label{sec:3}

In this section we discuss locally standard 2-torus actions on manifolds as their quotients will be nice manifolds with corners and then we will prove the quotient criterion for these particular actions analogously to \cite{franzQuot} for locally standard torus actions. Recall that we are considering cohomology with coefficients over a field $\bk$ of characteristic two and we will omit it in our notation.

We start by reviewing the standard action on $\R^m$. Let $G$ be a $2$ torus of rank $n$ with $n \leq m$. The \textit{standard action} of $G$ on $\R^m$ is defined as follows: Identifying $\Zd = \{ \pm 1 \}$, we have a canonical action of $G$ on $\R^m$ given by
\[ (g_1,\ldots, g_n) \cdot (x_1, \ldots, x_n, x_{n+1}, \ldots, x_m ) = (g_1x_1, \ldots, g_nx_n, x_{n+1}, \ldots, x_m) \]  
and thus the quotient space $\R^m / G \cong \R_{m,n} = [0,\infty)^n \times \R^{m-n}$ is a manifold with corners. This leads to the following definition.

\begin{definition}
	Let $G$ be a $2$-torus of rank $n$ and $X$ be a $G$-manifold of dimension $m$ with $m \geq n$. A \textit{$G$-standard chart} $(U, \varphi)$ of $x \in X$  is  a $G$-invariant open neighbourhood $U$ of $x$ in $X$  and a $G$-equivariant homeomorphism $\varphi \colon U \rightarrow V$ on some $G$-invariant open set $V \subseteq \R^m$ (with the standard action defined above).  We say that $X$ is a \textit{locally standard $G$-manifold} (or that the $G$-action is locally standard) if $X$ has an atlas $\{(U_i, \varphi_i)\}$ consisting of standard charts such that the change of coordinates $\varphi_i\circ \varphi_j^{-1} \colon \varphi_j(U_i \cap U_j) \rightarrow \varphi_i(U_i \cap U_j)$ is a $G$-equivariant diffeomorphism for all $i,j$.
\end{definition} 
Under these assumptions, one can check that the quotient space $M = X/G$ becomes an $m$-manifold with $n$-corners.  Let $\pi\colon X\rightarrow X/G$ denote the  quotient map. For any subspace $A \subseteq X/G$ we write $\pi^{-1}(A) = X^A$; however, we will identify $X^G$ with its image in $X/G$. 

Let $P$ be a face of $M$. Notice that all points in $X$ lying over the interior of $P$ have a common isotropy group $G_P \subseteq G$. We denote by $G_P^* = G/G_P$ which is isomorphic to a complementary $2$-torus of $G_P$ in $G$. Also, we write $\rank P = \rank G_P^*$. Observe that $X^P = \inv{\pi}(P)$ is a connected component of $X^{G_P}$ and the set $X^{P \setminus \partial P} = X^{P}\setminus X^{\partial P}$ is the open subset of $X^P$ where $G_P^*$ acts freely. This implies that 
\begin{equation}\label{equation423}
H_{G_P^*}(X^{P}, X^{\partial P}) \cong H^{*}(P, \partial P).
\end{equation}

For two faces $P,\  Q$ of $X/G$ we write $P \subseteq_1 Q$ if $P \subseteq Q$ and $\rank Q = \rank P +1$. If $F \subseteq_1 X/G$, then $X^F$ is a closed submanifold of $X$ of codimension $1$ and we denote by $N_F$ its normal bundle. For any face $P \subseteq F$, consider $E_{F.P}$ the vector bundle over $X^{P \setminus \partial P}$ obtained as the pullback of $N_F$ under the inclusion of $X^{P \setminus \partial P}$ on $X^F$.  Under the inclusion $(X^P,X^{\partial P}) \rightarrow (E_{F,P}, E_{F,P}^0)$, the equivariant Thom class of $E_{F,P}$ induces a class in the equivariant cohomology $e_{F,P} \in H_G^1(X^P, X^{\partial P})$. Finally, we denote by $t_{F,P} \in H^*(BG_p)$ the restriction of $e_{F, P}$  under the inclusion $(pt, \emptyset) \rightarrow (X^P, X^{\partial P})$. Notice that $t_{F,P}$ is the equivariant Euler class of the $G_p$-equivariant line bundle over a point which corresponds to the generator of  $H^*(BG_P)$ in $H^1(BG_P)$.

\begin{remark}
	For any face $P$ of $X/G$, consider the $\bk$-algebra $R_P  = H^*(BG_P)$. Then $P$ is a connected component of the intersection $\bigcap_{P \subseteq F \subseteq_1 X/G} F$ and thus $X^P$ is a connected component of the intersection $\bigcap_{P \subseteq F \subseteq_1 X/G} X^F$. Moreover, for any point $x \in X^{P\setminus \partial P}$, there is an isomorphism $G_P \cong \prod_{P \subseteq F \subseteq_1 X/G} G_F$ by looking to a standard chart of $x$ in $X$. This implies that $t_{F,P}$ is a basis for the vector space $H^1(BG_P)$ which extends to an isomorphism of algebras
	\[ R_P \cong \bk[t_{F,P} : P \subseteq F \subseteq_1 X/G ]\]
\end{remark}

If $P \subseteq Q$, $G_Q \subseteq G_P$ and we have a canonical map $\rho_{PQ} \colon R_P \rightarrow R_Q$. It follows from the naturality of the Euler class and the above remark that $\rho_{PQ}(t_{F,P}) = t_{F,Q} \;\text{if}\;\; Q\subseteq F$ and 0 otherwise. Now we will proceed to prove the following lemma.

\begin{proposition}\label{lemA5.2}
	Let $P$ be a face in $X/T$.
	\begin{enumerate}[(i)]
		\item The composition
		\[ \phi_P\colon H^*(P, \partial P) \xrightarrow{\cong} H^*_{G_P^*}(X^P, X^{\partial P}) \rightarrow H^*_G(X^P, X^{\partial P}) \]
		induces a map $\psi_P \colon H^*(P, \partial P) \otimes R_P \rightarrow H_G^*(X^P, X^{\partial{P}})$ which is an isomorphism of graded vector spaces.
		\item  If $P \subseteq_1 Q$ the following diagram
		\[
		\begin{tikzcd}
		H^*(P, \partial P) \otimes R_P \dar[swap]{\delta \otimes \rho_{PQ}} \rar{\psi_P} & H^*_G(X^P, X^{\partial P})\dar{\delta} \\
		H^{*+1}(Q, \partial Q)\otimes R_Q \rar{\psi_Q} & H^{*+1}_G(X^Q, X^{\partial{Q}}) 
		\end{tikzcd}
		\]
		is commutative where $\delta$ is the  connecting homomorphism arisen from the cohomology long exact sequence of the triple  $(Q, \partial Q, \partial Q \setminus (P \setminus \partial P) )$.
	\end{enumerate}
\end{proposition} 

\begin{proof}
	To prove the first claim, notice that the map \[\phi_P\colon H^*(P, \partial P) \xrightarrow{\cong} H^*_{G_P^*}(X^P, X^{\partial P}) \rightarrow H^*_G(X^P, X^{\partial P})\] is the composite of the isomorphism (\ref{equation423}) and the map in equivariant cohomology induced by the canonical projection $G \rightarrow G_P^* = G/G_P$. Suppose that $F_1, \ldots, F_k$ are the facets containing $P$. Using (i), we can define a map
	\[ \psi_P \colon H^*(P, \partial P) \otimes R_P \rightarrow H_G^*(X^P, X^{\partial{P}}) \]
	by setting $\psi_P(\alpha \otimes t_{F_1, P}^{m_1}\cdots t_{F_k, P}^{m_k}) = \phi_P(\alpha)e_{F_1,P}^{m_1} \cdots e_{F_k, P}^{m_k}$. On the other hand, we have an isomorphism of algebras
	\[ \rho \colon H_G^*(X^P, X^{\partial P}) \rightarrow R_P \otimes H_{G_P^*}(X^P, X^{\partial P}) \rightarrow H^*(P, \partial P) \otimes R_P  \]
	by choosing a splitting of $G = G_P \times G_P^*$. In particular, for $e_{F,P} \in H^1(X^P, X^{\partial P})$, we have that $\rho(e_{F,P}) \in H^0(P, \partial P) \otimes (R_P)_1 \oplus H^1(P, \partial P)\otimes(R_P)_0   \cong (R_P)_1 \oplus H^1(P, \partial P)$. As $t_{F,P}$ is the restriction of $e_{F,P}$ to $R_P$ we have then that $\rho(e_{F,P}) = t_{F,P} + a_F$ for some $a_F \in H^1(P, \partial P)$. Using this computation we get that for $\alpha \in H^*(P, \partial P)$ it holds that
	\begin{align*}
	\rho \circ \psi_P(\alpha \otimes t_{F_1, P}^{m_1}\cdots t_{F_k, P}^{m_k}) &= \rho(\phi_P(\alpha)e_{F_1,P}^{m_1} \cdots e_{F_k, P}^{m_k}) \\
	&= (\alpha \otimes 1)\rho(e_{F_1,P})^{m_1}\cdots \rho(e_{F_k, P})^{m_k}  \\
	&= (\alpha \otimes 1)(1 \otimes t_{F_1, P} + a_{F_1} \otimes 1)^{m_1} \cdots (1 \otimes t_{F_k, P}+ 1\otimes a_{F_k})^{m_k} \\
	&= \alpha \otimes (t_{F_1,P}^{m_1}\cdots t_{F_k,P}^{m_k}) + S
	\end{align*}
	where $S$ consists of sum of terms in $H^*(P, \partial P) \otimes H^*(R_P)$ whose elements in the second factor are of degree lower than $m_1 + \cdots + m_k$; therefore, we obtained that $\rho \circ \psi_P$ is bijective and so is $\psi_P$.
	
	\medskip
	
	Finally, to prove (iii), we need to check that $\delta \psi_P (\alpha \otimes t_{F_1,P}^{m_1}\cdots t_{F_k, P}^{m_k}) = \psi_Q(\delta(\alpha) \otimes \rho_{PQ} (t_{F_1,P}^{m_1} \otimes t_{F_k, P}^{m_k}))$. As the maps $\phi_P, \phi_Q$ arise from natural constructions, they commute with $\delta$. Furthermore,  since $\rho_{PQ}(t_{F,P})$ is either  $t_{F,Q}$ if $Q \subseteq F$ or zero otherwise, we only need to prove that $\delta(\beta e_{F,P})$ is either $\delta(\beta)e_{F,Q}$ if $Q \subseteq F$ or zero otherwise. Recall that $\delta$ arises from the connecting homomorphism  $\delta \colon H^*(P, \partial P) \cong H^*(\partial Q, \partial Q \setminus (P\setminus \partial P)) \rightarrow H^*(Q, \partial Q)$ which induces the map 
	\[ \delta \colon H^*_G(X^P, X^{\partial P}) \cong H_G^*(X^{\partial Q}, X^{\partial Q \setminus (P \setminus \partial Q)}) \rightarrow H_G^{*+1}(X^Q , X^{\partial Q}).\] 
	In the case $P \subseteq Q$, by the Thom isomorphism theorem we have isomorphisms $H^{*-1}_G(X^{P\setminus \partial P}) \cong H_G^*(X^P, X^{\partial P})$ and $H^{*-1}_G(X^{Q \setminus \partial Q}) \cong H^*_G(X^{Q}, X^{\partial Q})$ induced by the multiplication by $e_{F,P}$ and $e_{F,Q}$ respectively. As both $e_{F,P}$ and $e_{F,Q}$ are restrictions of the equivariant Euler class of the normal bundle $N^F$, we have that $\delta(\beta e_{F,P}) =  \delta(\beta)e_{F,Q}$. In the second case, we have that $e_{F,P}$ is then the restriction of the Euler class of the normal bundle of $X^P$ in $X^Q$ as $P \subseteq_1 Q$. By the Thom-Gysin exact sequence we have that $\delta$ vanishes precisely in the multiples of $e_{F,P}$. 
\end{proof}

For a face $P$ of $X/G$, the filtration by its faces leads to an spectral sequence with $E_1$-term given by \[E_1^{p,q}  = \bigoplus_{\substack{Q \subseteq P \\ \rank Q = i }} H^{p+q}(Q, \partial Q) \Rightarrow H^*(P) \]
the columns of this spectral sequence give rise to a complex that will be denoted by $B^{i}(P)$. This complex will be related to the Atiyah-Bredon sequence discussed at the beginning of this section  and it will provide a criterion to the syzygies in $G$-equivariant cohomology as it is shown in the following theorem which is analogous to \cite[Thm.1.3]{franzQuot} for the torus case.
\begin{theorem}\label{quotcrit}
	Let $X$ be a $G$-manifold with a locally standard action of a $2$-torus $G$. Then $H^*_G(X)$ is a $j$-th syzygy over $H^*(BG;\bk)$ if and only if for any face $P$ of the manifold with corners $M = X/G$ we have that $H^i(B^*(P))= 0$ for any $i > \max(\rank P - j, 0)$
\end{theorem}
\begin{proof}
	Let $Q$ be a face of $X/G$. We define the element $t_Q = \prod_{Q \subseteq F \subseteq_1 X/G} t_{F,Q} \in R_Q$. These elements induce an isomorphism of vector spaces $R_Q \cong \bigoplus_{Q \subseteq P} R_Pt_P $. On the other hand, by Proposition \ref{lemA5.2} there is an isomorphism of vector spaces $H^*(Q, \partial Q) \otimes R_Q \rightarrow H^*_G(X^Q, X^{\partial Q})$ compatible with the differentials. We have then an isomorphism
	\begin{equation}\label{equa455}
	\bigoplus_{Q: \rank Q = i} H^*(Q, \partial Q) \otimes R_Q \cong \bigoplus_{Q: \rank Q = i}H^*_G(X^{Q}, X^{\partial Q}).
	\end{equation}	
	
	Noticing that the $i$-th equivariant skeleton of $X$ is given by $\displaystyle X^{i} = \bigcup_{\substack{P\\ \rank P = i}} X^P = \bigcup_{\substack{P \\ \rank P = i+1}}
	X^{\partial P}$, we see that the last term of (\ref{equa455}) is the $i$-th term of the Atiyah-Bredon sequence $AB_G^{i}(X)$ and so there is an isomorphism (with an appropriate degree shift)
	\[ \bigoplus_{\substack{Q\\ \rank Q = i}} \; \bigoplus_{\substack{P\\ Q \subseteq P}}  H^*(Q,\partial Q) \otimes   R_Pt_P \cong \bigoplus_{P \subseteq X/G} B^{i}(P) \otimes R_Pt_P \cong AB_G^{i}(X) \]
	compatible with the differentials. Therefore, $H^i(AB^{*}_G(X)) = \bigoplus_{P \subseteq X/G} H^{i}(B^{*}(P)) \otimes R_Pt_P$.
	
	\medskip
	From Proposistion \ref{depthcor}, we have that $H^*_G(X)$ is a $j$-th syzygy if and only if $H^i(AB_{G_P^*}(X^P)) = 0$ for all faces $P$ and $i > \max(\rank P -j, 0)$. The isomorphism above shows that this condition is equivalent to the vanishing of  $H^i(B^*(P))$ for all $P$ and $i > \max(\rank P -j,0)$.
\end{proof}
We will use this criterion to construct syzygies in $G$-equivariant cohomology for $2$-torus actions. The dimension of a manifold with a locally standard action of  a $2$-torus is constrained to the rank of the torus. In fact, if $G = (\Zd)^r$ and $X$ is a $G$-manifold with a locally standard action of $G$ and $X^G \neq \emptyset$, then $\dim X \geq r$. In fact, if the action is locally standard, then $X^G$ is a submanifold of codimension at least $r$ and there can not be any fixed points if $\dim X < r$.
\begin{example}
	\normalfont
	If $X$ is a manifold with a locally standard action of $\Zd$, then the orbit space $M = X/G$ is a manifold with boundary. Conversely, any manifold with boundary can be realized as the orbit space of the manifold $X = (M \sqcup M)/ \partial M$ with the involution induced by the map $M \sqcup M \rightarrow M \sqcup M$ that swaps factors. The action is locally standard on $X$ as it can be seen as the reflection along the hyperplane where $ \partial M$ lies and so $X^G = \partial M$. 
	
	\medskip
	
	Theorem \ref{quotcrit} translates in this case on the statement that $X$ is $G$-equivariantly formal if and only if the map $H^*(\partial M) \rightarrow H^{*+1}(M, \partial M)$ is surjective, or equivalently, the map $H^*(M) \rightarrow H^*(\partial M)$ induced by the inclusion is injective. For example, if $M = S^1 \times [0,1]$ is a cylinder, then the map $H^*(M) \rightarrow H^*(\partial M)$ is injective and so the manifold $X$ is $G$-equivariantly formal. Then $X$ is homeomorphic to the torus $S^1 \times S^1$ and the involution is given by the axis reflection on one $S^1$-factor.  
	
	On the other hand, $M$ does not need to be orientable; for example, if $M$ is the Mobius strip, then $M$ can be realized as the orbit space of a Klein bottle $X$. Moreover, the map induced in cohomology by the inclusion $\partial M \rightarrow M$ is the zero map and thus Theorem \ref{quotcrit} implies that $H^*_G(X)$ is not equivariantly formal.
\end{example}

Let $G = \Zd \times \Zd$ and let $M$ be  a nice manifold with corners locally diffeomorphic to $\R^n \times [0, \infty)^2$. Then the face lattice of $M$ consists of facets $F$ (of rank 1) and faces $P$ of rank $0$. Suppose that $X$ is a $(n+2)$-manifold with a locally standard action of $G$. From Theorem \ref{quotcrit} we have the following cases.
\begin{itemize}
	\item $H_G^*(X)$ is a $2$-nd syzygy (or equivariantly formal) if and only if for any facet $F$ the map $H^*(\partial F) \rightarrow H^{*+1}(F, \partial F)$ is surjective and the sequence $\bigoplus_P H^*(P) \rightarrow \bigoplus_F H^{*+1}(F, \partial F) \rightarrow H^{*+2}(M, \partial M) \rightarrow 0$ is exact at the second and third position.
	
	\item $H_G^*(X)$ is a $1$-st syzygy if and only if for any facet $F$ the map $H^*(\partial F) \rightarrow H^{*+1}(F, \partial F)$ is surjective and the sequence $\bigoplus_P H^*(P) \rightarrow \bigoplus_F H^{*+1}(F, \partial F) \rightarrow H^{*+2}(M, \partial M) \rightarrow 0$ is exact only  at the third position. If the latter holds, the sequence is exact also at the second position. 
\end{itemize}

\begin{example}
	\normalfont
	Let $M$ be the $1$-simplex $\{ (x_1, x_2) \in \R_2, x_1 + x_2 \leq 1, x_1, x_2 \geq 0\}$. The manifold $X = \tilde{X}/\sim$ can be taken as the real projective space $\R P^2$ as shown in the following figure.


	\begin{center}
		\begin{multicols}{2}
			\begin{tikzpicture}
			\draw[-,blue, thick] (0,0) -- (2,0);
			\draw[-,violet, thick] (2,0) -- (2,2);
			\draw[-,yellow, thick] (2,2) -- (0,0);
			\draw[black] (1,-1) node[anchor=south] {};
			\draw [fill=gray!10] (0,0) -- (2,0) -- (2,2);
			\filldraw[black] (0,0) circle (2pt) ; 
			\filldraw[red] (2,2) circle (2pt); 
			\filldraw[magenta] (2,0) circle (2pt); 
			\end{tikzpicture}
			
			\columnbreak
			\begin{tikzpicture}
			\draw [fill=gray!10] (-1.5,-1.5) rectangle (1.5,1.5);
			\draw[-,teal, thick] (0,0) -- (-1.5,1.5);
			\draw[-,green, thick] (0,0) -- (-1.5,-1.5);
			\draw[-,yellow, thick] (-1.5,1.5) -- (-1.5,-1.5);
			\draw[-,yellow, thick] (1.5,1.5) -- (1.5,-1.5);
			\draw[-<,yellow, thick] (-1.5,0);
			\draw[->,yellow, thick] (1.5,0);
			\draw[-,blue, thick] (0,0) -- (1.5,-1.5);
			\draw[-,orange, thick] (-1.5,1.5) -- (1.5,1.5);
			\draw[-,orange, thick] (-1.5,-1.5) -- (1.5,-1.5);
			\draw[->,orange, thick] (0.1,-1.5) -- (0,-1.5);
			\draw[->,orange, thick] (0,1.5) -- (0.1,1.5);
			\draw[-,violet, thick] (0,0) -- (1.5,1.5);
			\filldraw[magenta] (0,0) circle (2pt) ; 
			\filldraw[red] (1.5,1.5) circle (2pt); 
			\filldraw[red] (-1.5,-1.5) circle (2pt); 
			\filldraw[black] (-1.5,1.5) circle (2pt);
			\filldraw[black] (1.5,-1.5) circle (2pt); 
			\end{tikzpicture}
			
		\end{multicols}
	\end{center}
	Since $M$ and its faces  are contractible it is easy to check that the $G$-equivariant cohomology of $X$ is a free module by looking at the complex $B^*(P)$ described above. Similarly, the action of $G$ on $X$ can be represented  as the reflection along the main diagonals on the square. Therefore, $X^G$ consists of $3$ points and thus $b(X) = b(X^G) = 3$ confirming the result obtained from the quotient criterion.

\end{example}

To construct a space whose equivariant cohomology is torsion-free but not free, we need to consider an action of a $2$-torus of rank at least $3$. Following \cite[Lemma 7.1]{F3}, let us start with the following manifold with corners
\[ M = \{ (u,z) \in ([0,\infty]\times \R^2)^3 : \vert z_i \vert^2 + \vert u_i \vert^2 = 1, u_1+u_2+u_3 = 0\}   \]
and $i=1,2,3$, where $\R_+$ denotes the non-negative real numbers. Then $M$ is a smooth manifold with corners locally diffeomorphic to $[0,\infty)^3 \times \R$. The projection $ M \rightarrow (\R^2)^3$ of the first component induces a homeomorphism between $M$ and the subspace of $(\R^2)^3$ consisting of these triples $(u_1,u_2, u_2)$ such that $\max\{ \vert u_1 \vert, \vert u_2 \vert, \vert u_3 \vert\} \leq 1$ and $u_1 + u_2 + u_3 = 0$. The latter space describes the configuration of triangles (including degenerate triangle) in $\mathbb{R}^2$ with sides of length at most 1. Therefore, $M$ is homeomorphic to the intersection of a $6$-dimensional ball with a linear subspace of codimension $2$ and thus $M$ is topologically a $4$-dimensional ball. In particular, $\partial M \cong S^3$ and $H^*(M, \partial M) \cong \widetilde{H}^*(S^4)$. Now we will look at the face decomposition of $M$.

\begin{itemize}
	\item $M$ has exactly one face $P$ of rank zero. Namely, it is given by those elements $(u,z) \in M$ such that $z_i = 0$, and then $u_i \in S^1$ for all $i$. Since one of the $u's$ entries depends on the other two, $P$  can be identified with the manifold
	\[ P = \{(u_1,u_2,u_3) \in (S^1)^3 : u_1+u_2+u_3 = 0\} = \{ (x,y) \in S^1 \times S^1 : \vert x-y \vert = 1 \} \]
	$P$ then is the configuration space of equilateral triangles in $\R^2$ with one vertex in the origin and two over the circle. Each of these configurations is determined by a rotation of any of the pairs $(1, e^{i\pi/3})$ or $(1, e^{-i\pi/3})$. In particular, this implies that $P \cong S^1 \sqcup S^1$. Thus we have that $H^0(P) \cong H^1(P) = \bk \oplus \bk$ and it is zero in any other degree.
	
	\item $M$ has three faces of rank $1$. Namely, $Q_{12}, Q_{13}$ and $Q_{23}$ where $Q_{ij}$ consists of the pairs $(u,z) \in M$ such that $z_i = z_j = 0$.  We identify $Q_{ij}$ with the manifold with boundary
	
	\[ Q = \{ (x,y) \in S^1 \times S^1 : \vert x-y \vert \leq 1 \} \]
	
	In terms of configuration spaces, this consists of isosceles triangles with one vertex in the origin, two over the circle and whose base is of length at most $1$ (Here we allow the degenerate triangle). We can show that there is a homeomorphism $Q \cong S^1 \times I$ given by a rotation of the pairs $(1,e^{it \pi/3}) \in Q$ where $-1 \leq t \leq 1$. Computing the relative cohomology $H^*(Q, P)$  of the cylinder relative to the boundary we see that $H^1(Q,P) \cong H^2(Q,P) \cong \bk$ and it is zero in other degrees.
	
	\item $M$ has three facets (of rank 2). Namely, $F_{1}, F_2, F_3$ where $F_i$ consists of the pairs $(u,z) \in M$ such that $z_i = 0$.  We identify $F_i$ with the manifold with corners
	
	\[ F = \{ (x,y) \in S^1 \times D_2 : \vert x-y \vert  \leq  1 \} \]
	
	This space describes the configuration of triangles with one side of length $1$, and two of length at most $1$. Each of these configurations is determined by a rotation of the pairs $(1, se^{it \pi/3}) \in F$ where $ 0 \leq s \leq 1$ and $-1 \leq t \leq 1$. Then $F$ is homeomorphic to $S^{1} \times I \times I \cong S^1 \times D_2$. Looking at the relative cohomology $H^*(F,\partial F)$ of the solid torus with respect to its boundary (the torus) we find that $H^2(F, \partial F) \cong H^3(F, \partial F) \cong \bk$  and it is zero in other degrees.

\end{itemize}

The face lattice of $M$ is then

\begin{center}
	\begin{tikzpicture}
	\draw[-,black] (0,0) -- (-1,1);
	\draw[-,black] (0,0) -- (0,1);
	\draw[-,black] (0,0) -- (1,1);
	\draw[-,black] (-1,1.5) -- (-1.5,2.5);
	\draw[-,black] (-1,1.5) -- (-0.1,2.5);
	\draw[-,black] (0,1.5) -- (-1.4,2.5);
	\draw[-,black] (0,1.5) -- (1.4,2.5);
	\draw[-,black] (1,1.5) -- (1.5,2.5);
	\draw[-,black] (1,1.5) -- (0,2.5);
	\draw[-,black] (0,3) -- (0,4);
	\draw[-,black] (-1,3) -- (0,4);
	\draw[-,black] (1,3) -- (0,4);
	\draw[black] (0,-0.5) node[anchor=south] {$P$};
	\draw[black] (-1,1.5) node[anchor=north] {$Q_{12}$};
	\draw[black] (0,1.5) node[anchor=north] {$Q_{13}$};
	\draw[black] (1,1.5) node[anchor=north] {$Q_{23}$};
	\draw[black] (-1.5,3) node[anchor=north] {$F_{1}$};
	\draw[black] (0,3) node[anchor=north] {$F_{2}$};
	\draw[black] (1.5,3) node[anchor=north] {$F_{3}$};
	\draw[black] (0,4.5) node[anchor=north] {$M$};
	\end{tikzpicture}
\end{center}
\medskip

Consider the manifold
\[ X = \{ (z,u) \in (\R \times \R^2)^3 : \vert z_i \vert^2 + \vert u_i \vert^2 = 1, u_1+u_2+u_3 = 0\}   \]
with the locally standard action of $G$ on $X$ given by multiplication on the variables $z_i$. Then $X$ is  $G$-locally standard manifold and $X/G \cong M$. We will see that the $G$-equivariant cohomology of $X$ is a first syzygy but not a second syzygy using Theorem \ref{quotcrit}. It is a first syzygy as the maps

\[ H^*(P) \rightarrow H^{*+1}(Q,P)\]\[  H^{*}(Q_{j},P) \oplus H^*(Q_k, P) \rightarrow H^{*+1}(F_i, \partial F_i), \]
\[ \bigoplus_{i=1}^3 H^*(F_i, \partial F_i) \rightarrow H^{*+1}(M, \partial M) \]

are surjective as it can be seen by using the explicit computation of these groups mentioned above. On the other hand, $H^*_G(X)$ is not a second syzygy as the sequence
\[ \bigoplus_{j=1}^3 H^*(Q_j, \partial Q_j)\rightarrow \bigoplus_{i=1}^3 H^{*+1}(F_i, \partial F_i) \rightarrow H^{*+2}(M, \partial M) \rightarrow 0\]
is not exact at the second position; that is, $H^2(B^*(M)) \neq 0$. In fact, the complex $B^*(M)$ takes the form 
\[  \bk^3 \rightarrow \bk^3 \rightarrow 0 \rightarrow  0 \]
when $\ast = 1$. The map $ \bk^3 \rightarrow \bk^3$ is given by $(a,b,c) = (a+b, a+c, b+c)$ which is of rank 2  and then  $H^2(B^*(M)) \neq 0$. 

\medskip

We constructed a $4$-dimensional manifold $X$ with an action of $G = (\Zd)^3$  such that the equivariant cohomology $H^*_G(X)$ is torsion-free but not free as $H^*(BG)$-module. The manifold $X$ realizes the smallest possible dimension where a manifold  with a locally standard action of a $2$-torus $G$ whose equivariant cohomology is torsion-free but not free exists. As we previously discussed, if $\rank G \leq 2$ then being free is equivalent to being torsion-free in equivariant cohomology, so the minimal example should occur when $\rank G = 3$. On the other hand, if the dimension of a manifold is the same as the rank of the torus, its $G$-equivariant cohomology is free if and only if it is torsion-free \cite[\S 7.2]{franzQuot}; therefore, we must have that $\dim X > 3$.

\medskip
The generalization of this space is discussed in \cite{P} which is the real version of the big Polygon spaces introduced in \cite{F3}. These spaces allow us to  construct syzygies of higher order in equivariant cohomology for 2-torus actions \cite{P} and also for torus actions with a compatible involution  \cite{Ch}.

\bibliographystyle{abbrv}
\bibliography{2tori}

\end{document}